\documentclass[10pt,a4paper]{article}
\usepackage[utf8]{inputenc}
\usepackage[english]{babel}
\usepackage{amsmath}
\usepackage{amsthm}
\usepackage{amsfonts}
\usepackage{amssymb}

\providecommand{\keywords}[1]
{
	\textbf{\textit{Keywords---}} #1
}
\usepackage{graphicx}

\usepackage[margin=2.5cm]{geometry}

\usepackage{color}

\theoremstyle{definition}
\newtheorem{theorem}{Theorem}[section]

\newtheorem{lemma}{Lemma}[section]

\newtheorem{remark}{Remark}[section]

\allowdisplaybreaks

\author{Fausto Colantoni}

\begin{document}

\title{Non-local skew and non-local sticky Brownian motions}

\maketitle
\begin{abstract}
	In this paper, we present a comprehensive study on the generalization of skew Brownian motion and two-sided sticky Brownian motion by considering non-local operators at the origin for the heat equations on the real line. To begin, we introduce Marchaud-type operators and Caputo-Dzherbashian-type operators, providing an in-depth exposition of their fundamental properties. Subsequently, we describe the two stochastic processes and the associated equations. The non-local skew Brownian motion exhibits jumps, as a subordinator, at zero where the sign of the jump is determined by a skew coin. Conversely, the non-local sticky Brownian motion displays stickiness at zero, behaving as the inverse of a subordinator, resulting in non-Markovian dynamics.
\end{abstract}

\keywords{Skew Brownian motions, sticky Brownian motions, non-local operators}
\section{Introduction}
Skew Brownian motion and sticky Brownian motion are two of the most well-known variations of the standard Brownian motion. The first one allows the skewness that introduces asymmetry, causing the process to have a preference for either upward or downward movements \cite{walsh1978diffusion, shepp, lejay}. The second one tends to stick to certain levels or boundaries. It is characterized by periods of slow movement or delayed diffusion near certain values or regions \cite{bass-sticky,sde-sticky,sticky-numerical}.
Both skew Brownian motions and sticky Brownian motions have found applications in various fields. In finance, skew Brownian motions have been employed to model asset price movements and option pricing \cite{arbitrage}. Sticky Brownian motions, on the other hand, have been used to describe phenomena where particles exhibit sticking behavior, for example in  adsorption models for molecules \cite{graham}.

In this paper, we generalize these two process by introducing non-local operators at a certain level. In particular, the space non-local operators lead to jumps, where there can be a preference for jumping above or below this level, while non-local operators in time slow down the process, making the dynamics no longer Markovian. The idea is to leverage the recent works on non-local boundary value problems \cite{mirko-fbvp1,mirko-fbvp2,coldov2022halfline}, where now instead of the boundary, a certain threshold is considered 
so that the dynamics are anomalous. In this new topic of probability theory, the authors deal with non-local operators (deriving from fractional derivatives such as Caputo-Dzherbashian or Marchaud derivatives) as boundary conditions. A possible reference is \cite{dovidio2022}.

Anomalous diffusions and fractional calculus are often used for real-world applications, such as in finance, physics and hydrodynamics. The literature on the subject is extensive, see, for example, \cite{benson2000application, scalas-mainardi}. The present work involves a Brownian particle displaying anomalous behavior, characterized by jumps or slowdowns at a specific level. Such behavior can be applied to situations where anomalies are associated with a particular threshold.
\section{Basic concepts}
A subordinator is a non-decreasing L\'evy process (\cite[Chapter III]{bertoin1996levy}) and its Laplace exponent is given by a Bernstein function (\cite[Theorem 5.1]{Schilling}).\\
Let us introduce the Bernstein function $\Phi$, defined by the L\'evy-Khintchine representation (\cite[Theorem 3.2]{Schilling})
\begin{align}
\label{LevKinFormula}
\Phi(\lambda)= \int_0^\infty (1-e^{-\lambda z}) \Pi^\Phi(dz), \quad \lambda>0,
\end{align}
where  $\Pi^\Phi$ a L\'evy measure on $(0,\infty)$ such that $\int_0^\infty (1 \wedge z) \Pi^\Phi(dz) < \infty$. We introduce the subordinator $H^\Phi$ which is
characterized by the Laplace exponent $\Phi$, that is
\begin{align}
\label{LapH}
\mathbf{E}_0[\exp(-\lambda H_t^\Phi)]=\exp(-t\Phi(\lambda)), \quad \lambda > 0,
\end{align}
where we denote by $\mathbf{E}_x$ the expected value with respect to $\mathbf{P}_x$ where $x$ is the starting point. Since $\Pi^\Phi(0,\infty)=\infty$, then, from \cite[Theorem 21.3]{ken1999levy}, we have that $H^\Phi$ has strictly increasing sample paths with jumps, indeed the symbol $\Phi$ does not admit any drift.\\
We also define the inverse process $L^\Phi=\{L_t^\Phi,\ t \geq 0\}$, with $L_0^\Phi=0$,  that is the right inverse of $H^\Phi$
\begin{align*}
L_t^\Phi = \inf \{s > 0\,:\, H_s^\Phi >t \}, \quad t>0.
\end{align*}
Because $H^\Phi$ is strictly increasing,  the inverse process $L^\Phi$ turns out to be a continuous process. In particular, in correspondence of the jumps of $H^\Phi$, the process $L^\Phi$ has intervals of consistency (plateaus). Notice that, an inverse process can be regarded as an exit time for $H^\Phi$. By definition, we also have
\begin{align}
\label{relationHL}
\mathbf{P}_0(H_t^\Phi < s) = \mathbf{P}_0(L_s^\Phi>t), \quad s,t>0.
\end{align}
We recall that
\begin{align}
\label{tailSymb}
\frac{\Phi(\lambda)}{\lambda}= \int_0^\infty e^{-\lambda z} \overline{\Pi}^\Phi(z) dz, \quad \overline{\Pi}^\Phi(z)=\Pi^\Phi(z,\infty)
\end{align}
where $\overline{\Pi}^\Phi$ is the so called \textit{tail of the L\'evy measure $\Pi^\Phi$}.

For us, $B$ denotes the standard one dimensional Brownian motion with law $g(t, z) = e^{-z^2/ 4t}/\sqrt{4 \pi t}$, independent of $H^\Phi$. Let us consider the heat equation on the positive half-line with a Neumann boundary condition at $x=0$, that is
\begin{align*}
\begin{cases}
\frac{\partial}{\partial t} u(t,x)= \frac{\partial^2}{\partial x^2} u(t,x) \quad &(t,x) \in (0,\infty) \times (0,\infty)\\
\frac{\partial}{\partial x} u(t,x) \big\vert_{x=0}=0\quad \quad \quad \quad \quad &t>0
\end{cases}.
\end{align*}
The probabilistic representation of the solution can be written in terms of the reflecting Brownian motion $B^+=\{B_t^+, t \geq 0\}$ (see for example \cite{itomckean-halfline} or \cite[Lemma 6.3]{essentials}). We indicate by $\gamma_t=\gamma_t(B^+)$ the local time at zero of the reflecting Brownian motion $B^+$. 
We now define the process $B^\bullet$ as
\begin{align}
\label{Bpallino}
B_t^{\bullet} = H^\Phi L^\Phi \gamma_t - \gamma_t + B_t^+, \quad t>0,
\end{align}
where we use the following notation
\begin{align*}
H^\Phi L^\Phi \gamma_t=H^\Phi \circ L^\Phi \circ \gamma_t
\end{align*}
for the composition of $ H^\Phi, L^\Phi ,\gamma$. The process \eqref{Bpallino} was introduced in \cite{itomckean-halfline} to study the heat equation with the integral boundary condition presented in \cite{feller}. It was established that $B^\bullet$ is a strong Markov process and its local time at zero is $L^\Phi \gamma$ \cite[Section 14]{itomckean-halfline}. The sample paths of $B^\bullet$ were described in \cite[Section 12]{itomckean-halfline} as a reflected Brownian motion, which exhibits random jumps away from the origin in the positive half-line. These jumps correspond to the last jump of $H^\Phi$. It is important to note that $B^\bullet$ is a right-continuous process due to the fact that $H^\Phi L^\Phi$ is the composition of the right-continuous subordinator $H^\Phi$ with its inverse $L^\Phi$, which is a continuous process.
	\begin{figure}[h]
	\centering
	\includegraphics[width=10cm]{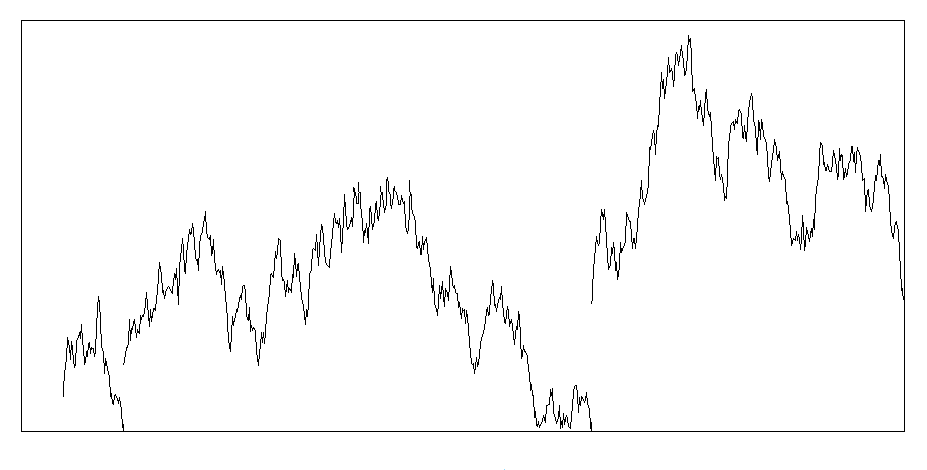} 
	\caption{A possible path for $B^\bullet$.}
	\label{fig:Bpallino}
\end{figure}
\section{Skew Brownian motions and sticky Brownian motions}
The skew Brownian motion, introduced in \cite{itomckean-paths, walsh1978diffusion}, is a straightforward extension of standard Brownian motion on the real line. In this process, the sign of each excursion is determined through independent skew coin-tossing, introducing a level of asymmetry. A possible definition is the following (\cite[Section 4.2, Problem 1]{itomckean-paths}):  $J_1, J_2,...$ denote a fixed enumeration of the excursion intervals of the reflected process $B^+$. For a given parameter $0 < \nu < 1$, let $\{A_m : m = 1, 2, . . . \}$ be an i.i.d. sequence of Bernoulli $\pm1$ valued random variables, independent of $B^+$, also defined on the same probability space with $\mathbf{P}(A_i=1)=\nu$. Define $\nu -$skew
Brownian motion process $B^\nu = \{B^\nu _t: t\geq 0 \}$ by
\begin{align}
\label{skew}
B^\nu _t = \displaystyle \sum_{m=1}^{\infty} \mathbf{1}_{J_m}(t) A_m B^+_t,
\end{align}
for which we have, \cite[Proposition 1 and Proposition 2]{lejay}, that the probabilistic representation of the solution $u \in C^{1,2}((0,\infty) \times \mathbb{R}\setminus \{0\}; \mathbb{R})$ and continuous in $x=0$ of the problem
\begin{align*}
\begin{cases}
\frac{\partial}{\partial t} u(t,x)= \frac{\partial^2}{\partial x^2} u(t,x) \quad &(t,x) \in (0,\infty) \times \mathbb{R}\setminus \{0\} \\
\nu \frac{\partial}{\partial x} u(t,0^+)= (1-\nu) \frac{\partial}{\partial x} u(t,0^-) \quad    &t>0\\
u(t,0^+)=u(t,0^-)\\
u(0,x) =f(x) \quad &x \in \mathbb{R}
\end{cases}
\end{align*}
with $f$ continuous and bounded and $\nu \in [0,1]$, is given by
\begin{align*}
u(t,x)=\mathbf{E}_x [f(B^\nu_t)].
\end{align*}
The process $B^\nu$ is the unique strong solution of the stochastic differential equation (\cite{shepp})
\begin{align*}
X_t=x + B_t+ (2\nu -1) \ell_t^0(X),
\end{align*}
where $\ell_t^0(X)$ is the symmetric local time of $X$ at zero. 
 For a complete discussion on its density, occupation and local times see \cite{occupation}.
 
Ont the other hand, the one side sticky Brownian motion is characterized by the Feller-Wentzell boundary conditions (\cite[Section 10]{itomckean-halfline}):
\begin{align}
\label{bc: Feller-Wentzell}
p_3 \frac{\partial^2}{\partial x^2} u(t,0)= p_2 \frac{\partial}{\partial x} u(t,0)
\end{align}
and the solution of the heat equation with these boundary conditions is written as the time change $B^+_{A_t^{-1}}=B^+ \circ A_t^{-1}$, where $A_t^{-1}$ is the right inverse of
\begin{align*}
A_t= t+ \frac{p_3}{p_2}\gamma_t.
\end{align*}
Because of this time change, the process is delayed in zero and then it reflects inside the positive half-line. If we extend the heat equation up to the boundary, the condition \eqref{bc: Feller-Wentzell} can be seen as a dynamic boundary condition, but we will delve into this concept later.\\
In the present paper we focus on a generalization of the two-sided sticky Brownian motion, where, by introducing the skewness, the particle in zero is delayed and it reflects in one of the two possible directions with a certain probability (see \cite[Section 17]{itomckean-halfline}). For a comprehensive treatment of these two sticky processes, please also refer to \cite{howitt2007stochastic}.


The motivation to explore non-local conditions at zero stems from several observations. For instance, when extending the heat equation up to zero in \eqref{bc: Feller-Wentzell} and considering dynamic conditions (by replacing $\frac{\partial^2}{\partial x^2}$ with $\frac{\partial}{\partial t}$), the question arises: What if we substitute the temporal derivative with a different type of operator that generalizes time derivatives? A similar inquiry can be made regarding spatial derivatives in conditions at zero for skew Brownian motion. Specifically, we aim to understand whether the process exhibits anomalous behavior near zero. To address this, we introduce non-local operators and employ them as conditions within the heat equation.

\section{Non-local operators}
In the current section we introduce the non-local operators which we will use as jumping or delaying conditions at zero. We use the common notation for Sobolev spaces 
\begin{align*}
W^{1,p}(\mathbb{R}) = \left\{ u \in L^p(\mathbb{R}):\;  u^\prime \in L^p(\mathbb{R}) \right\} \quad p \in [1,\infty].
\end{align*}
Then, for $u \in W^{1,\infty}(\mathbb{R})$ we define the left Marchaud-type derivative
\begin{align}
\label{Marchaudright}
\mathbf{D}_{x-}^\Phi u(x)=\int_0^\infty (u(x) - u(x-y))\Pi^\Phi(dy), \quad -\infty < x < +\infty
\end{align}
and the right  Marchaud-type derivative
\begin{align}
\label{Marchaudleft}
\mathbf{D}_{x+}^\Phi u(x) = \int_0^\infty (u(x) - u(x+y))\Pi^\Phi(dy),  \quad -\infty < x < +\infty.
\end{align}
Indeed, we see that
\begin{align}
\label{stimaMarchaud}
\vert \mathbf{D}_{x+}^\Phi u(x) \vert &\leq \int_0^1 \vert u(x) -u(x+y) \vert \Pi^\Phi(dy) + \int_1^\infty \vert u(x) -u(x+y) \vert \Pi^\Phi(dy) \notag \\
&\leq K \int_0^1 y \, \Pi^\Phi(dy) + 2 \vert\vert u \vert \vert_\infty \int_1^\infty \Pi^\Phi(dy) \notag \\
&\leq(K+2 \vert \vert u \vert \vert_\infty) \int_0^\infty (1 \wedge y) \Pi^\Phi(dy) <\infty,
\end{align}
where for the first integral we use that the function $u$ is locally Lipschitz, since its first derivative is in $L^\infty(\mathbb{R})$, while for the second integral we have utilized the fact that $u \in L^\infty(\mathbb{R})$. The same is true for the other derivative \eqref{Marchaudright}.\\
We call the operators  \eqref{Marchaudright} and \eqref{Marchaudleft} \textit{Marchaud-type derivatives} because they are a generalization of the Marchaud derivatives on the real line (see \cite[Formula (5.57) and (5.58)]{samko1993fractional}), that are the case of $\Phi(\lambda)=\lambda^\alpha$, $\alpha \in (0,1)$, and $\Pi^\Phi(dy)=\frac{\alpha}{\Gamma(1-\alpha)} \frac{dy}{y^{\alpha+1}}$.

As for the Marchaud-type operators, we consider a generalization, through different L\'evy measures,  of the Caputo-Dzherbashian derivative (see \cite[Section 2.4]{kilbas}).
Let $M>0$, $w\geq 0$ and $\mathcal{M}_w$ be the set of (piecewise) continuous function on $[0, \infty)$ of exponential order $w$ such that $\vert \varphi(t)\vert \leq M e^{wt}$. Let $\varphi \in \mathcal{M}_w \cap C([0,\infty))$ with $\varphi^\prime \in \mathcal{M}_w$. Then we define the Caputo-Dzherbashian-type derivative as
\begin{align}
\label{Caputo}
\mathfrak{D}_{t}^\Phi \varphi(t):=\int_0^t \varphi^\prime (s) \overline{\Pi}^\Phi(t-s) ds
\end{align}
which is a convolution type operator. Indeed, by using \eqref{tailSymb}, the Laplace transform is
\begin{align}
\label{LapCaputo}
\int_0^\infty e^{-\lambda t}  \mathfrak{D}_{t}^\Phi \varphi(t) dt= \Phi(\lambda) \widetilde{\varphi}(\lambda)- \frac{\Phi(\lambda)}{\lambda} \varphi(0), \quad \lambda > w
\end{align}
where $\widetilde{\varphi}$ is the Laplace transform of $\varphi$. For a more recent overview on this operator, we recommend reading \cite{kochubei2011general,toaldo2015convolution,chen2017time}.
\section{Non-local skew Brownian motions}
In \cite{coldov2022halfline}, the authors provide that the process $B^\bullet$ is the probabilistic solution of the heat equation on $(0,\infty)$ with a Marchaud-type derivative at zero, as \eqref{Marchaudleft} when it is restricted to the positive half-line. So, the jumps of this process are due to those of the subordinator, which by its nature is connected to the Marchaud operator.

We want now to generalize the process \eqref{skew} for the Marchaud-type derivatives by inducing jumps. As for the $\nu-$ skew Brownian motion, let us define the $\nu-$ skew $B^\bullet$. Let $J_1^\bullet, J_2^\bullet,...$ be a fixed enumeration of the excursion intervals of the process $B^\bullet$. As before, $\{A_m : m = 1, 2, . . . \}$ is an i.i.d. sequence of Bernoulli $\pm1$ valued random variables, independent of $B^\bullet$, also defined on the same probability space with $\mathbf{P}(A_i=1)=\nu$. Define $\nu -$skew
Brownian motion process $B^{\nu \bullet} = \{B^{\nu \bullet}_t: t\geq 0 \}$ by
\begin{align}
\label{skew-Bpallino}
B^{\nu \bullet} _t = \displaystyle \sum_{m=1}^{\infty} \mathbf{1}_{J_m^\bullet}(t) A_m B^\bullet_t.
\end{align}
The process $B^{\nu \bullet}$  in zero with probability $\nu$ jumps as $B^\bullet$, otherwise like $-B^\bullet$. Since $B^\bullet$ is a right-continuous strong Markov process, $B^{\nu \bullet}$ is also a right-continuous strong Markov process.
\begin{remark}
	The excursions $J_i$ and $J_i^\bullet$, $i \geq 1$, are not the same. Therefore, even though $B^\bullet$ behaves like $B^+$ outside of zero, since $B^\bullet$ restarts after a jump of the subordinator $H^\Phi$, it takes more time to reach zero compared to $B^+$, which reflects immediately.
\end{remark}
Before proving the next theorem, let us introduce the space we need for the non local conditions at zero
\begin{align*}
D_1:=\{\varphi: \forall t>0 ,\, \varphi(t, \cdot) \in W^{1,\infty} (\mathbb{R}) \text{ s.t. } \mathbf{D}_{x+}^\Phi \varphi(t,x) \big\vert_{x=0^+}  \text{ and } \mathbf{D}_{x-}^\Phi \varphi(t,x) \big\vert_{x=0^-} \text{ exist} \}
\end{align*}
\begin{theorem}
	\label{thm:non-local-skew}
	The probabilistic representation of the solution $u \in C^{1,2}((0,\infty) \times \mathbb{R} \setminus \{0\}) \cap D_1$ and continuous in $x=0$, for $\nu \in (0,1)$ and $f$ continuous and bounded,
	\begin{align*}
	\begin{cases}
	\frac{\partial}{\partial t} u(t,x)= \frac{\partial^2}{\partial x^2} u(t,x) \quad &(t,x) \in (0,\infty) \times \mathbb{R}\setminus \{0\} \\
	\nu \, \mathbf{D}_{x+}^\Phi u(t,x) \big\vert_{x=0^+} + (1-\nu) \, \mathbf{D}_{x-}^\Phi u(t,x) \big\vert_{x=0^-}=0\quad    &t>0\\
	u(0,x) =f(x) \quad &x \in \mathbb{R}
	\end{cases}
	\end{align*}
	is the following
	\begin{align*}
	u(t,x)= \mathbf{E}_x\left[f(B^{\nu \bullet} _t)\right]
	\end{align*}
\end{theorem}
\begin{proof}
	Let us focus on the resolvent of $u$, for $\lambda >0$,
	\begin{align*}
	R_\lambda f(x)&= \int_0^\infty e^{-\lambda t} u(t,x) dt\\
	&= \int_{0}^{\infty} e^{-\lambda t} \mathbf{E}_x\left[f(B^{\nu \bullet} _t)\right] \, dt.
	\end{align*}
	We define the hitting time
	\begin{align*}
	\tau_0:=\inf \{t>0: B_t =0\}
	\end{align*}
	and since, before hitting zero, $B^{\nu \bullet}$ behaves like $B$, $\tau_0$ is also the hitting time at zero of $B^{\nu \bullet}$. Now, we have
	\begin{align*}
	R_\lambda f(x) &= \int_{0}^{\infty} e^{-\lambda t} \mathbf{E}_x\left[f(B^{\nu \bullet} _t)\right] \, dt\\
	&= \int_{0}^{\tau_0} e^{-\lambda t} \mathbf{E}_x\left[f(B^{\nu \bullet} _t)\right] \, dt + \int_{\tau_0}^{\infty} e^{-\lambda t} \mathbf{E}_x\left[f(B^{\nu \bullet} _t)\right] \, dt\\
	&=  \int_{0}^{\infty} e^{-\lambda t} \mathbf{E}_x\left[f(B^{D} _t)\right] \, dt + \mathbf{E}_x[e^{-\lambda \tau_0}] \int_{0}^{\infty} e^{-\lambda t} \mathbf{E}_0\left[f(B^{\nu \bullet} _t)\right] \, dt,
	\end{align*}
	where we have used that the process before $\tau_0$ is a Dirichlet Brownian motion $B^D$, killed at zero, and the strong Markov property for $B^{\nu \bullet}$. We recall that, for $\lambda>0$, (see \cite[Lemma 2.11]{essentials})
	\begin{align*}
	\mathbf{E}_x[e^{-\lambda \tau_0}]= e^{-\sqrt{\lambda} \vert x \vert}
	\end{align*}
and (see \cite[Example 7.14]{schilling-brownian})
\begin{align}
\label{lapg1}
\int_0^\infty e^{-\lambda t} g(t,x) dt &= \frac{1}{2} \frac{e^{-\vert x \vert \sqrt{\lambda}}}{\sqrt{\lambda}}\\
\label{lapg2}
\int_0^\infty e^{-\lambda t} \frac{x}{t} g(t,x) dt &=
e^{-\vert x \vert \sqrt{\lambda}}.
\end{align}
Then, by inverting the Laplace transform, we obtain that the solution is written as
\begin{align*}
u(t,x)= \int_{x\cdot y>0} \left(g(t,x-y) - g(t,x+y)\right)f(y) dy + \int_0^t \frac{x}{s} g(s,x) u(t-s,0) ds,
\end{align*}
where we can use the Dirichlet kernel only if $x$ and $y$ have the same sign ($x \cdot y>0$), otherwise, from the fact that the paths, outside from zero, are continuous, we can not reach $y$ without passing for $0$. The choice of the initial datum $f$ continuous and bounded allows us to obtain a bounded semigroup, as in the case of \cite[Proposition 1]{lejay}. The resolvent turns out to be
\begin{align*}
R_\lambda f(x) &= \int_{0}^{\infty} e^{-\lambda t} \int_{x\cdot y>0} \left(g(t,x-y) - g(t,x+y)\right)f(y) dy\, dt + e^{-\sqrt{\lambda} \vert x \vert} \int_{0}^{\infty} e^{-\lambda t} \mathbf{E}_0\left[f(B^{\nu \bullet} _t)\right] \, dt\\
&= \int_{0}^{\infty} e^{-\lambda t} \int_{x\cdot y>0} \left(g(t,x-y) - g(t,x+y)\right)f(y) dy\, dt + e^{-\sqrt{\lambda} \vert x \vert} R_\lambda f(0).
\end{align*}
Since the first part is the Dirichlet kernel for the heat equation, for $x \ne 0, \lambda >0$ we have
\begin{align*}
\frac{\partial^2}{dx^2} R_\lambda f(x)= \lambda R_\lambda f(x)-f(x),
\end{align*}
then we get 
\begin{align*}
\frac{\partial}{dt} u(t,x)= \frac{\partial^2}{\partial x^2} u(t,x) \quad (t,x) \in (0,\infty) \times \mathbb{R}\setminus \{0\},
\end{align*}
as requested. Now, we move on the conditions at zero. First, by using the definition of $B^{\nu \bullet}$ we observe that, for $\lambda >0$,
\begin{align}
\label{resolvent-skew-zero}
R_\lambda f(0)&= \int_{0}^{\infty} e^{-\lambda t} \mathbf{E}_0\left[f(B^{\nu \bullet} _t)\right] \, dt \notag\\
&=\mathbf{E}_0\left[\int_0^\infty e^{-\lambda t} \sum_{n=1}^{\infty} \mathbf{1}_{J_n^\bullet}(t) f(A_n B_t^\bullet) dt\right] \notag\\
&= \sum_{n=1}^\infty \mathbf{E}_0 \left[\int_{J_n^\bullet} e^{-\lambda t}f(A_n B^\bullet_t) dt\right] \notag\\
&=\nu\sum_{n=1}^\infty \mathbf{E}_0 \left[\int_{J_n^\bullet} e^{-\lambda t}f(B_t^\bullet) dt\right] +(1-\nu) \sum_{n=1}^\infty \mathbf{E}_0 \left[\int_{J_n^\bullet} e^{-\lambda t}f(- B_t^\bullet) dt\right] \notag\\
&=\nu \mathbf{E}_0 \left[\int_{0}^\infty e^{-\lambda t}f(B_t^\bullet) dt\right] +(1-\nu) \mathbf{E}_0 \left[\int_{0}^\infty e^{-\lambda t}f(- B_t^\bullet) dt\right] \notag \\
&=\begin{cases}
\nu \frac{\int_0^\infty R_\lambda^D f(y)   \, \Pi^\Phi (dy)}{\Phi(\sqrt{\lambda})} \quad y \geq 0\\
\\(1-\nu) \frac{\int_0^\infty R_\lambda^{-D} f(-y)  \, \Pi^\Phi (dy)}{\Phi(\sqrt{\lambda})} \quad y <0,
\end{cases}
\end{align}
where, in the last equality, we have used that $B^\bullet$ has positive trajectories and
\begin{align}
\label{resolvent:Bpallino}
\mathbf{E}_0 \left[\int_{0}^\infty e^{-\lambda t}f(B_t^\bullet) dt\right]= \frac{\int_0^\infty R_\lambda^D f(y) \, \Pi^\Phi (dy)}{\Phi(\sqrt{\lambda})}
\end{align}
which is provided in \cite[Formula 6, Section 15]{itomckean-halfline}, where $R_\lambda^D f(y)$ is the resolvent of the killed Brownian motion $B^D$ and $R_\lambda^{-D} f(-y)$ the resolvent of $-B^D$. To facilitate the reader, we include the proof of \eqref{resolvent:Bpallino} in the Appendix.

Now, we provide the boundary conditions. For $y>0$, from the linearity of the Marchaud-type operators, we have
\begin{align*}
	&\nu \, \mathbf{D}_{x+}^\Phi R_\lambda f(x) \big\vert_{x=0^+}= \nu \left(\mathbf{D}_{x+}^\Phi R_\lambda^D f(x) + \nu \frac{\int_0^\infty R_\lambda^D f(y) \, \Pi^\Phi (dy)}{\Phi(\sqrt{\lambda})} \mathbf{D}_{x+}^\Phi e^{-\sqrt{\lambda} x}\right)_{x=0^+}\\
	&=\nu \left(\lim_{x\downarrow 0}  \int_0^\infty \left( R^D_\lambda f(x) - R^D_\lambda f(x+y) \right) \Pi^\Phi(dy)+ \nu \frac{\int_0^\infty R_\lambda^D f(y) \, \Pi^\Phi (dy)}{\Phi(\sqrt{\lambda})} \lim_{x\downarrow 0}  \int_0^\infty \left( e^{-x \sqrt{\lambda}} - e^{-(x+y) \sqrt{\lambda}} \right) \Pi^\Phi(dy)\, \right)\\
	&=\nu \left(- \int_0^\infty R^D_\lambda f(y)\, \Pi^\Phi(dy) +\nu \frac{\int_0^\infty R_\lambda^D f(y) \, \Pi^\Phi (dy)}{\Phi(\sqrt{\lambda})} \int_{0}^\infty (1-e^{-\sqrt{\lambda} y }) \Pi^\Phi(dy) \right)\\
	&=\nu \left(- \int_0^\infty R^D_\lambda f(y)\, \Pi^\Phi(dy) +\nu \frac{\int_0^\infty R_\lambda^D f(y) \, \Pi^\Phi (dy)}{\Phi(\sqrt{\lambda})} \Phi(\sqrt{\lambda})\right)\\
	&=\nu \left(- \int_0^\infty R^D_\lambda f(y)\, \Pi^\Phi(dy) +\nu \int_0^\infty R_\lambda^D f(y) \, \Pi^\Phi (dy)\right)\\
	&=\nu (\nu-1) \int_0^\infty R_\lambda^D f(y) \, \Pi^\Phi (dy).
\end{align*}
On the other hand (so when $x$ is negative), for $y>0$ we have
\begin{align*}
(1-\nu) \, \mathbf{D}_{x-}^\Phi R_\lambda f(x) \big\vert_{x=0^-}&=(1-\nu) \nu \frac{\int_0^\infty R_\lambda^D f(y) \, \Pi^\Phi (dy)}{\Phi(\sqrt{\lambda})} \mathbf{D}_{x-}^\Phi e^{\sqrt{\lambda} x}\big\vert_{x=0^-} \\
&=(1-\nu) \nu \frac{\int_0^\infty R_\lambda^D f(y) \, \Pi^\Phi (dy)}{\Phi(\sqrt{\lambda})} \lim_{x\uparrow 0}  \int_0^\infty \left( e^{x \sqrt{\lambda}} - e^{(x-y) \sqrt{\lambda}} \right) \Pi^\Phi(dy)\\
&=(1-\nu) \nu \frac{\int_0^\infty R_\lambda^D f(y) \, \Pi^\Phi (dy)}{\Phi(\sqrt{\lambda})} \int_{0}^\infty (1-e^{-\sqrt{\lambda} y }) \Pi^\Phi(dy)\\
&= (1-\nu) \nu \frac{\int_0^\infty R_\lambda^D f(y) \, \Pi^\Phi (dy)}{\Phi(\sqrt{\lambda})} \Phi(\sqrt{\lambda})\\
&=(1-\nu) \nu {\int_0^\infty R_\lambda^D f(y) \, \Pi^\Phi (dy)},
\end{align*}
where we have exploited that $x$ and $y$ have opposite signs. Then we obtain, for $y>0$ and $\lambda>0$
\begin{align}
\label{bc:skew-pos}
\nu \, \mathbf{D}_{x+}^\Phi R_\lambda f(x) \big\vert_{x=0^+} + (1-\nu) \, \mathbf{D}_{x-}^\Phi  R_\lambda f(x)\big\vert_{x=0^-}=0.
\end{align}
We use the same argument for negative trajectories. Indeed, for $y<0$ and $x<0$, we have
\begin{align*}
&(1-\nu) \mathbf{D}_{x-}^\Phi R_\lambda f(x) \big\vert_{x=0^-}=(1-\nu) \left(\mathbf{D}_{x-}^\Phi R_\lambda^{-D} f(-x) + (1-\nu) \frac{\int_0^\infty R_\lambda^{-D} f(-y) \, \Pi^\Phi (dy)}{\Phi(\sqrt{\lambda})} \mathbf{D}_{x-}^\Phi e^{\sqrt{\lambda} x}\right)_{x=0^-}\\
&=(1-\nu) \left(\lim_{x\uparrow 0}  \int_0^\infty \left( R^{-D}_\lambda f(-x) - R^{-D}_\lambda f(-x-y) \right) \Pi^\Phi(dy)\right)\\
&+ (1-\nu)\left((1-\nu)\frac{\int_0^\infty R_\lambda^{-D} f(-y) \, \Pi^\Phi (dy)}{\Phi(\sqrt{\lambda})} \lim_{x\uparrow 0}  \int_0^\infty \left( e^{x \sqrt{\lambda}} - e^{(x-y) \sqrt{\lambda}} \right) \Pi^\Phi(dy)\, \right)\\
&=(1-\nu) \left(- \int_0^\infty R^{-D}_\lambda f(-y)\, \Pi^\Phi(dy) +(1-\nu) \frac{\int_0^\infty R_\lambda^{-D} f(-y) \, \Pi^\Phi (dy)}{\Phi(\sqrt{\lambda})} \int_{0}^\infty (1-e^{-\sqrt{\lambda} y }) \Pi^\Phi(dy) \right)\\
&=(1-\nu) \left(- \int_0^\infty R^{-D}_\lambda f(-y)\, \Pi^\Phi(dy) +(1-\nu) \frac{\int_0^\infty R_\lambda^{-D} f(-y) \, \Pi^\Phi (dy)}{\Phi(\sqrt{\lambda})} \Phi(\sqrt{\lambda})\right)\\
&=(1-\nu) \left(- \int_0^\infty R^{-D}_\lambda f(-y)\, \Pi^\Phi(dy) +(1-\nu) \int_0^\infty R_\lambda^{-D} f(-y) \, \Pi^\Phi (dy)\right)\\
&=-(1-\nu) \nu \int_0^\infty R_\lambda^{-D} f(-y) \, \Pi^\Phi (dy).
\end{align*}
Similarly, for $x>0$ and $y<0$, we get
\begin{align*}
\nu \, \mathbf{D}_{x+}^\Phi R_\lambda f(x) \big\vert_{x=0^+}&=\nu (1-\nu) \frac{\int_0^\infty R_\lambda^{-D} f(-y) \, \Pi^\Phi (dy)}{\Phi(\sqrt{\lambda})} \mathbf{D}_{x+}^\Phi e^{-\sqrt{\lambda} x}\big\vert_{x=0^+} \\
&=\nu (1-\nu) \frac{\int_0^\infty R_\lambda^{-D} f(-y) \, \Pi^\Phi (dy)}{\Phi(\sqrt{\lambda})} \lim_{x\downarrow 0}  \int_0^\infty \left( e^{-x \sqrt{\lambda}} - e^{-(x+y) \sqrt{\lambda}} \right) \Pi^\Phi(dy)\\
&=\nu (1-\nu) \frac{\int_0^\infty R_\lambda^{-D} f(-y) \, \Pi^\Phi (dy)}{\Phi(\sqrt{\lambda})} \int_{0}^\infty (1-e^{-\sqrt{\lambda} y }) \Pi^\Phi(dy)\\
&= \nu (1-\nu) \frac{\int_0^\infty R_\lambda^{-D} f(-y) \, \Pi^\Phi (dy)}{\Phi(\sqrt{\lambda})} \Phi(\sqrt{\lambda})\\
&=\nu (1-\nu) {\int_0^\infty R_\lambda^{-D} f(-y) \, \Pi^\Phi (dy)}.
\end{align*}
Then, again, we obtain, for $y<0$ and $\lambda>0$
\begin{align}
\label{bc:skew-neg}
\nu \, \mathbf{D}_{x+}^\Phi R_\lambda f(x) \big\vert_{x=0^+} + (1-\nu) \, \mathbf{D}_{x-}^\Phi  R_\lambda f(x)\big\vert_{x=0^-}=0.
\end{align}
By inverting the Laplace transform in \eqref{bc:skew-pos} and \eqref{bc:skew-neg}, since we are dealing with a continuous function, we obtain that
\begin{align*}
\nu \, \mathbf{D}_{x+}^\Phi u(t,x) \big\vert_{x=0^+} + (1-\nu) \, \mathbf{D}_{x-}^\Phi u(t,x) \big\vert_{x=0^-}=0\quad    &t>0.
\end{align*}
\end{proof}
	\begin{figure}[h!]
	\centering
	\includegraphics[width=8cm]{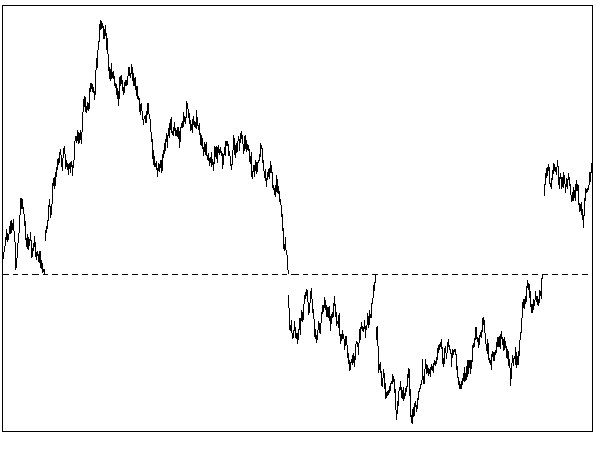} 
	\caption{A possible path for $B^{\nu \bullet}$.}
	\label{fig:non-local-skew}
\end{figure}
As illustrated in Figure \ref{fig:non-local-skew}, the behavior of the process $B^{\nu \bullet}$ closely resembles that of $B^\bullet$, but with a notable distinction at zero. Here, with a probability of $\nu$, it jumps as $B^\bullet$, otherwise, it emulates the negative jump of $-B^\bullet$. Moreover, at the origin, it jumps as the last jump of the subordinator  $H^\Phi$, with both the sign and the continuation of the stochastic process determined through the outcome of a coin toss with a skew distribution.

We want to emphasize that the roles of the left and right derivative, in the case of the skew Brownian motion, are now taken by the two Marchaud-type derivatives \eqref{Marchaudright} and \eqref{Marchaudleft}, which therefore appear as a generalization of the concept of derivative. Let us delve deeper into this idea with the following remark.
\begin{remark}
	We want to investigate in what sense the operators \eqref{Marchaudright} and \eqref{Marchaudleft} and the process $B^{\nu \bullet}$ are a generalization of derivatives and skew Brownian motion. Let $f$ be in the Schwartz space $\mathcal{S}(\mathbb{R})$ and $\Phi(\lambda)=\lambda^\alpha$, with $\alpha \in (0,1)$, be the Bernstein function associated to the $\alpha-$stable subordinator. Then, the operators $\mathbf{D}_{x-}^\alpha$ and  $\mathbf{D}_{x+}^\alpha$ are the Marchaud derivatives on the real line and, for $\xi \in \mathbb{R}$, we have
	\begin{align*}
	\int_{-\infty}^{\infty} e^{-i \xi x} \mathbf{D}_{x-}^\alpha f(x)dx&= 	\int_{-\infty}^{\infty} e^{-i \xi x} \int_0^\infty (f(x) - f(x-y))\Pi^\alpha(dy) dx\\
	&=\widehat{f}(\xi) \int_{0}^{\infty} (1-e^{-i\xi y}) \Pi^\alpha(dy)\\
	&=(i\xi)^\alpha \widehat{f}(\xi),
	\end{align*}
	where $\widehat{f}(\xi)$ is the Fourier transform of $f$. Similarly, we get
		\begin{align*}
	\int_{-\infty}^{\infty} e^{-i \xi x} \mathbf{D}_{x+}^\alpha f(x)dx&= 	\int_{-\infty}^{\infty} e^{-i \xi x} \int_0^\infty (f(x) - f(x+y))\Pi^\alpha(dy) dx\\
	&=\widehat{f}(\xi) \int_{0}^{\infty} (1-e^{i\xi y}) \Pi^\alpha(dy)\\
	&=(-i\xi)^\alpha \widehat{f}(\xi).
	\end{align*}
	When $\alpha \uparrow 1$, we have
	\begin{align*}
	\int_{-\infty}^{\infty} e^{-i \xi x} \mathbf{D}_{x-}^\alpha f(x)dx=(i\xi)^\alpha \widehat{f}(\xi) &\rightarrow (i \xi) \widehat{f}(\xi)= 	\int_{-\infty}^{\infty} e^{-i \xi x} \frac{\partial}{\partial x} f(x) dx\\
	\int_{-\infty}^{\infty} e^{-i \xi x} \mathbf{D}_{x+}^\alpha f(x)dx=(-i\xi)^\alpha \widehat{f}(\xi) &\rightarrow (-i \xi) \widehat{f}(\xi)= 	\int_{-\infty}^{\infty} e^{-i \xi x} \left(-\frac{\partial }{\partial x} f(x)\right) dx.
	\end{align*}
	In this sense, we can say that we have convergence to first derivatives and for $\alpha \uparrow 1$ the conditions become
	\begin{align*}
		\nu \, \mathbf{D}_{x+}^\alpha f(x) \big\vert_{x=0^+} + (1-\nu) \, \mathbf{D}_{x-}^\alpha f(x) \big\vert_{x=0^-}=0 \rightarrow - \nu \frac{\partial}{\partial x} f(x)\big\vert_{x=0^+} + (1-\nu) \frac{\partial}{\partial x} f(x) \big\vert_{x=0^-}=0,
	\end{align*}
	which coincide with those of the skew Brownian motion. We also notice that the fact that the process $B^{\nu \bullet}$, in the limit $\alpha \uparrow 1$, behaves like a skew Brownian motion is evident from its trajectories. In fact, for the limit $\alpha \uparrow 1$, we observe that the $\alpha-$ stable subordinator tends to $t$, then the process $B^\bullet$ does not exhibit jumps.
\end{remark}
\begin{remark}
It is possible to generalize the Fourier transforms seen in the previous remark, provided we take into account the measures $\Pi^\Phi$, and we have, for $f \in \mathcal{S}(\mathbb{R})$ and $\xi \in \mathbb{R}$,
\begin{align*}
\int_{-\infty}^{\infty} e^{-i \xi x} \mathbf{D}_{x-}^\Phi f(x) dx &=\Phi(i\xi) \widehat{f}(\xi)\\
\int_{-\infty}^{\infty} e^{-i \xi x} \mathbf{D}_{x+}^\Phi f(x) dx &=\Phi(-i\xi) \widehat{f}(\xi).
\end{align*}
We know that the same Fourier transform can be obtained from other operators, such as Weyl-type derivatives or generalized Riemann-Liouville derivatives on the real line (see \cite[Lemma 2.9]{toaldo2015convolution}). In this paper, we have chosen to focus on Marchaud-type derivatives for historical reasons as well. In fact, as noted in \cite{coldov2022halfline}, the operator $-\mathbf{D}_{x+}^\Phi$, restricted on $x \in \mathbb{R}^+$, coincides with 
\begin{align*}
\int_0^\infty (f(y) - f(0)) \Pi^\Phi(dy),
\end{align*}
which is the Feller integral condition introduced in \cite{feller}.
\end{remark}
\section{Non-local sticky Brownian motions}
In \cite{mirko-fbvp2}, the author presents a delayed Brownian motion at the boundary, as the inverse of a subordinator, related to fractional conditions. Now, by including a skewed coin toss, we take that same process to generalize the two-sided sticky Brownian motion.

First we introduce the one side sticky Brownian motion $B^s$, that is a generalization of \cite[Section 4]{mirko-fbvp2} or a special case of \cite{coldov2022halfline},
\begin{align}
\label{def:Bsticky}
B^s_t = B^+ \circ T_t^{-1}=B^+_{T_t^{-1}},
\end{align}
where $T_t^{-1}$ is the right of
\begin{align*}
T_t= t+H^\Phi \circ \eta \gamma_t,
\end{align*}
where $\eta>0$ is a positive constant, $H^\Phi$ is the subordinator, related to the symbol $\Phi$, independent of the reflecting Brownian motion $B^+$, and $\gamma$ is the local time at zero of $B^+$. The process $B^s$ is generated by the heat equation with Caputo-Dzherbashian-type derivatives at the boundary (the operator defined in \eqref{Caputo}), in particular
\begin{align*}
\left(\eta \mathfrak{D}^\Phi_t u (t,x) -\frac{\partial}{\partial x} u(t,x) \right)_{x=0}=0, \quad t>0,
\end{align*}
for more details on this type of non-local dynamic conditions, please refer to the works mentioned above. It is a non-Markov process, since in zero it has intervals of consistency (plateaus), due to the inverse of the subordinator $H^\Phi$. After these plateaus, it reflects in the positive half line as $B^+$. To construct the two-sided sticky process, we first need the following result.
	\begin{figure}[h!]
	\centering
	\includegraphics[width=8cm]{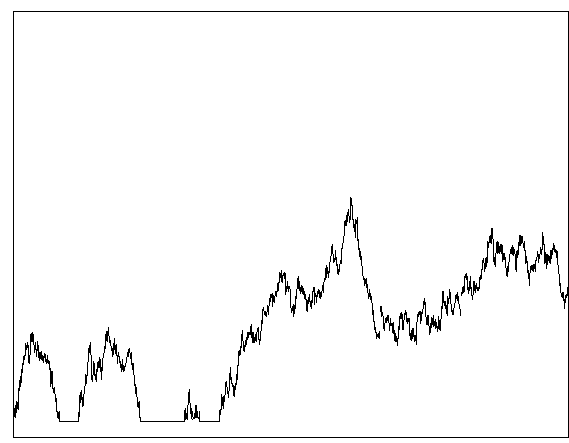} 
	\caption{A possible paths for $B^{s}$.}
	\label{fig:sticky-bm}
\end{figure}
\begin{lemma}
	For $f$ continuous and bounded and $\lambda >0$, we have
	\begin{align}
	\label{resolvent: B^s}
	\mathbf{E}_0\left[\int_0^\infty e^{-\lambda t} f(B^s_t) dt\right]=\frac{1}{\eta \Phi(\lambda) + \sqrt{\lambda}}\int_0^\infty e^{-y \sqrt{\lambda}} f(y) dy + \frac{\eta \Phi(\lambda)}{\lambda} \frac{f(0)}{\eta \Phi(\lambda) +\sqrt{\lambda}}
	\end{align}
\end{lemma}
\begin{proof}
	We have that
	\begin{align*}
\mathbf{E}_0\left[\int_0^\infty e^{-\lambda t} f(B^s_t) dt\right]&=\mathbf{E}_0\left[\int_0^\infty e^{-\lambda t} f(B^+_{T_t^{-1}}) dt\right]\\
&=\mathbf{E}_0\left[\int_0^\infty e^{-\lambda T_t} f(B^+_t)\, dT_t\right]\\
&=\mathbf{E}_0\left[\int_0^\infty e^{-\lambda t} e^{-\lambda H^\Phi_{\eta \gamma_t}} f(B^+_t)\, dt\right] + \mathbf{E}_0\left[\int_0^\infty e^{-\lambda t} e^{-\lambda H^\Phi_{\eta \gamma_t}} f(B^+_t)\, d(H^\Phi_{\eta \gamma_t})\right]\\
&=: I_1 + I_2.
	\end{align*}
	For $I_1$ we recall \eqref{LapH} and the joint Laplace transform of $B^+$ and its local time (the joint law is given in \cite[Section 4, formula 10]{itomckean-halfline})
	\begin{align}
	\label{joint-laplace}
	\int_{0}^{\infty} e^{-\lambda t} \mathbf{P}_0(B^+_t \in dy, \gamma_t \in dw)=e^{-(y+w) \sqrt{\lambda}} dy \, dw , \quad \lambda >0,
	\end{align}
	then we obtain
	\begin{align*}
	I_1&=\mathbf{E}_0\left[\int_0^\infty e^{-\lambda t} e^{-\lambda H^\Phi_{\eta \gamma_t}} f(B^+_t)\, dt\right]\\
	&=\mathbf{E}_0\left[\int_0^\infty e^{-\lambda t} e^{-\lambda \Phi(\lambda) {\eta \gamma_t}} f(B^+_t)\, dt\right]\\
	&=\int_{0}^{\infty} \int_{0}^{\infty} f(y) e^{-w \Phi(\lambda) \eta} e^{-(y+w) \sqrt{\lambda}} dy \, dw\\
	&=\frac{1}{\eta \Phi(\lambda) + \sqrt{\lambda}}\int_0^\infty e^{-y \sqrt{\lambda}} f(y) dy.
	\end{align*}
	For $I_2$, by integrating by parts and by using \eqref{LapH}, we get
	\begin{align*}
	I_2&=\mathbf{E}_0\left[\int_0^\infty e^{-\lambda t} e^{-\lambda H^\Phi_{\eta \gamma_t}} f(B^+_t)\, d(H^\Phi_{\eta \gamma_t})\right]\\
	&=f(0) \mathbf{E}_0\left[\int_0^\infty e^{-\lambda t} e^{-\lambda H^\Phi_{\eta \gamma_t}} d(H^\Phi_{\eta \gamma_t})\right]\\
	&=f(0) \mathbf{E}_0\left[\int_0^\infty e^{-\lambda  \gamma^{-1}_t} e^{-\lambda H^\Phi_{\eta t}} d(H^\Phi_{\eta t})\right]\\
	&=\frac{f(0)}{\lambda} \left(-1-\int_{0}^{\infty} e^{-\lambda H^\Phi_{\eta t}} d(e^{-\lambda\gamma^{-1}_t})\right)\\
	&=\frac{f(0)}{\lambda} \left(-1-\int_{0}^{\infty} e^{-\eta \Phi(\lambda) t} d(e^{-\lambda \gamma^{-1}_t})\right)\\
	&=\frac{f(0)}{\lambda} \left(-1-	\left(-1-\eta\Phi(\lambda)\int_{0}^{\infty} e^{-\eta \Phi(\lambda)t} e^{-\lambda\gamma^{-1}_t} dt\right)\right)\\
	&=\frac{f(0)}{\lambda} \eta\Phi(\lambda) \int_{0}^{\infty} e^{-\eta \Phi(\lambda)t} e^{-\lambda\gamma^{-1}_t} dt.
	\end{align*}
	The last integral we have to calculate is the following
	\begin{align*}
	 \int_{0}^{\infty} e^{-\eta \Phi(\lambda)t} e^{-\lambda\gamma^{-1}_t} dt&= \int_{0}^{\infty} e^{-\eta \Phi(\lambda)\gamma_t} e^{-\lambda t} d\gamma_t\\
	 &=-\frac{1}{\eta \Phi(\lambda)} \int_{0}^{\infty} e^{-\lambda t} d  e^{-\eta \Phi(\lambda)\gamma_t}\\
	 &=-\frac{1}{\eta \Phi(\lambda)}\left(-1+\lambda\int_{0}^{\infty} e^{-\lambda t} e^{-\eta \Phi(\lambda)\gamma_t} dt  \right),
	\end{align*}
	by integrating \eqref{joint-laplace} in $dy$, we obtain
	\begin{align*}
-\frac{1}{\eta \Phi(\lambda)}\left(-1+\lambda\int_{0}^{\infty} e^{-\lambda t} e^{-\eta \Phi(\lambda)\gamma_t} dt  \right)&=-\frac{1}{\eta \Phi(\lambda)}\left(-1+\lambda\int_{0}^{\infty} \frac{1}{\sqrt{\lambda}}e^{-\sqrt{\lambda} w} e^{-\eta \Phi(\lambda)w} dw \right)\\
&=-\frac{1}{\eta \Phi(\lambda)} \left(-1+\frac{\sqrt{\lambda}}{\sqrt{\lambda} + \eta \Phi(\lambda)}\right)\\
&=\frac{1}{\eta \Phi(\lambda) + \sqrt{\lambda}}.
	\end{align*}
	Then, we provide
	\begin{align*}
	I_2=\frac{\eta \Phi(\lambda)}{\lambda} \frac{f(0)}{\eta \Phi(\lambda) +\sqrt{\lambda}},
	\end{align*}
	and the claim \eqref{resolvent: B^s} follows from $I_1+I_2$.
\end{proof}
\begin{remark}
	We observe that $a(t):= H^\Phi_{\gamma_t}$ is a right continuous and nondecreasing process starting from zero, then from \cite[Lemma 2.2, section V]{blumenthal}, for every nonnegative Borel measurable function $f$ on the positive half-line vanishing at infinity,
	\begin{align*}
	\int_{(0,\infty)} f(t) d a(t)= \int_0^\infty f(a^{-1} (t)) dt,
	\end{align*}
	where $a^{-1} (t)$ is the right inverse of $a(t)$. In this sense we interpret the integrals of the last lemma.
\end{remark}
The process $B^s$ has the same excursions on $(0,\infty)$ of the process $B^+$, since it is a reflecting Brownian motion delayed in zero, but now the zeros have not Lebesgue measure zero. So, in a construction like that of skew Brownian motion, we now have to take into account the plateaus as well.

Let us define the two-sided non-local sticky Brownian motion. Let $J_1, J_2,...$ be a fixed enumeration of the excursion intervals of the process $B^+$ and $I_1, I_2,...$ be the subsequent plateaus (intervals of consistency) of $B^s$. As before, $\{A_m : m = 1, 2, . . . \}$ is an i.i.d. sequence, independent of $B^+$ and $H^\Phi$, of Bernoulli $\pm1$ valued random variables also defined on the same probability space with $\mathbf{P}(A_i=1)=\nu$. Then, the $\nu -$sticky
Brownian motion process $B^{\nu, s} = \{B^{\nu, s}_t: t\geq 0 \}$ is defined as
\begin{align}
\label{2sideBs}
B^{\nu , s} _t = \displaystyle \sum_{m=1}^{\infty} \mathbf{1}_{\{J_m \cup I_m\}}(t) A_m B^s_t.
\end{align}
The process $B^{\nu , s}$ is delayed in zero, after each plateau, it can reflect positively or negatively due to the Bernoulli variable. The difference with sticky Brownian motion is that $B^{\nu , s}$ is slowed down by the inverse of a subordinator, which makes the process non-Markovian.
	\begin{figure}[h!]
	\centering
	\includegraphics[width=8cm]{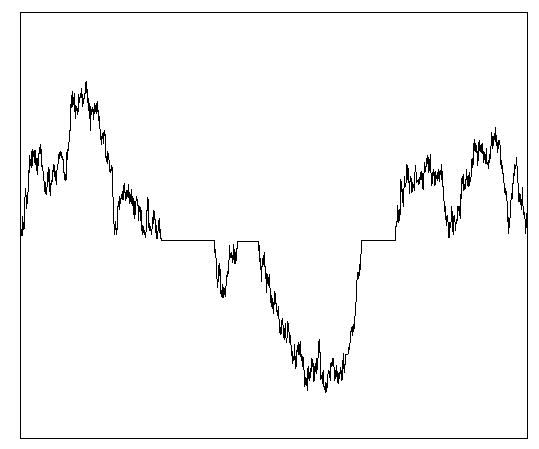} 
	\caption{A possible path for $B^{\nu , s}$.}
	\label{fig:non-local-sticky-bm}
\end{figure}
We introduce the following space functions, for the boundary conditions,
\begin{align*}
D_2:=\{\varphi: \forall x \in \mathbb{R} , \, \varphi(\cdot, x) \in W^{1,\infty} (0,\infty) \text{ s.t. }\mathfrak{D}_t^\Phi \varphi(t,x) \big\vert_{x=0} \text{ exists} \}.
\end{align*}
\begin{theorem}
\label{thm:non-local-sticky}
	The probabilistic representation of the solution $u \in C^{1,2}([0,\infty) \times \mathbb{R} \setminus \{0\}) \cap D_2$ and continuous in $x=0$, for $\nu \in (0,1)$ and $f$ continuous and bounded,
\begin{align*}
\begin{cases}
\frac{\partial}{\partial t} u(t,x)= \frac{\partial^2}{\partial x^2} u(t,x) \quad &(t,x) \in (0,\infty) \times \mathbb{R}\setminus \{0\} \\
\eta \mathfrak{D}_t^\Phi u(t,0) - \nu \frac{\partial}{\partial x} u(t,0^+) + (1-\nu )\frac{\partial}{\partial x} u(t,0^-)=0\quad    &t>0\\
u(0,x) =f(x) \quad &x \in \mathbb{R}
\end{cases}
\end{align*}
is the following
\begin{align*}
u(t,x)= \mathbf{E}_x\left[f(B^{\nu, s} _t)\right].
\end{align*}
\end{theorem}
\begin{proof}
	Let us focus on the Laplace transform of $u$, for $\lambda >0$,
	\begin{align*}
	R_\lambda f(x)&= \int_0^\infty e^{-\lambda t} u(t,x) dt\\
	&= \int_{0}^{\infty} e^{-\lambda t} \mathbf{E}_x\left[f(B^{\nu, s} _t)\right] \, dt.
	\end{align*}
	We recall that
	\begin{align*}
	\tau_0=\inf \{t>0: B_t =0\}
	\end{align*}
	and since, before hitting zero, $B^{\nu,s}$ behaves like $B$, $\tau_0$ is also the hitting time at zero of $B^{\nu,s}$.
	 The problem with respect to Theorem \ref{thm:non-local-skew} is that now we are dealing with a non-Markovian process, and we cannot use the same arguments as before. We decompose the positive half-line in the union of $\{J_m\}$ and $\{I_m\}$, with $m \geq 1$, as in \eqref{2sideBs}. But on every interval $J_m$ and on every interval $I_m$, the process is strongly Markovian, in fact, in the first case, it behaves like a Brownian motion, while in the second case, it is constant. So even though it is not globally Markovian, it is Markovian on all these intervals. We have
	 \begin{align*}
	 R_\lambda f(x) &= \int_{0}^{\infty} e^{-\lambda t} \mathbf{E}_x\left[f(B^{\nu, s} _t)\right] \, dt\\
	 &= \int_{0}^{\tau_0} e^{-\lambda t} \mathbf{E}_x\left[f(B^{\nu, s} _t)\right] \, dt + \int_{\tau_0}^{\infty} e^{-\lambda t} \mathbf{E}_x\left[f(B^{\nu, s} _t)\right] \, dt\\
	 &=  \int_{0}^{\infty} e^{-\lambda t} \mathbf{E}_x\left[f(B^{D} _t)\right] \, dt + \int_{I_1} e^{-\lambda t} \mathbf{E}_x\left[f(B^{\nu, s} _t)\right] \, dt + \int_{J_2} e^{-\lambda t} \mathbf{E}_x\left[f(B^{\nu, s} _t)\right] \, dt + \dots\\
	 &=\int_{0}^{\infty} e^{-\lambda t} \mathbf{E}_x\left[f(B^{D} _t)\right] \, dt  + \mathbf{E}_x[e^{-\lambda \tau_0}] \int_{0}^{\infty} e^{-\lambda t} \mathbf{E}_0\left[f(B^{\nu, s} _t)\right] \, dt,
	 \end{align*}
	 where we have used the strong Markov property in each interval and the fact that the process $B^{\nu, s}$, when starts from zero, has immediately the plateau $I_1$. With the same arguments of Theorem \ref{thm:non-local-skew}, we write the solution as
	 \begin{align*}
	 u(t,x)= \int_{x\cdot y>0} \left(g(t,x-y) - g(t,x+y)\right)f(y) dy + \int_0^t \frac{x}{s} g(s,x) u(t-s,0) ds,
	 \end{align*}
	 and, once again, it solves the heat equation in which we only need to check the conditions at zero. The time Laplace transform of the solution is
	 \begin{align*}
	 R_\lambda f(x) &= \int_{0}^{\infty} e^{-\lambda t} \int_{x\cdot y>0} \left(g(t,x-y) - g(t,x+y)\right)f(y) dy\, dt + e^{-\sqrt{\lambda} \vert x \vert} \int_{0}^{\infty} e^{-\lambda t} \mathbf{E}_0\left[f(B^{\nu,s } _t)\right] \, dt\\
	 &= \int_{0}^{\infty} e^{-\lambda t} \int_{x\cdot y>0} \left(g(t,x-y) - g(t,x+y)\right)f(y) dy\, dt + e^{-\sqrt{\lambda} \vert x \vert} R_\lambda f(0).
	 \end{align*}
	 The conditions we need have the following Laplace transform
	 \begin{align}
	 \label{laplace:BC}
	 &\int_0^\infty e^{-\lambda t} \left(\eta \mathfrak{D}_t^\Phi u(t,0) - \nu \frac{\partial}{\partial x} u(t,0^+) + (1-\nu )\frac{\partial}{\partial x} u(t,0^-)\right) dt \notag\\
	 &= \eta \Phi(\lambda) R_\lambda f(0) - \eta \frac{\Phi(\lambda)}{\lambda} f(0) - \nu \frac{\partial}{\partial x} R_\lambda f(x) \big\vert_{x=0^+} +(1-\nu) R_\lambda f(x) \big\vert_{x=0^-},
	 \end{align}
	 where we used \eqref{LapCaputo} for the Laplace transform of the non-local time operator. Now, let us understand how $R_\lambda f(0)$ is composed. For $\lambda >0$, from the skewness of the process, we get
	 \begin{align}
	 \label{resolvent-sticky-zero}
	 R_\lambda f(0)&= \int_{0}^{\infty} e^{-\lambda t} \mathbf{E}_0\left[f(B^{\nu,s} _t)\right] \, dt \notag\\
	 &=\mathbf{E}_0\left[\int_0^\infty e^{-\lambda t} \sum_{n=1}^{\infty} \mathbf{1}_{I_n \cup J_n}(t) f(A_n B_t^s) dt\right] \notag\\
	 &= \sum_{n=1}^\infty \mathbf{E}_0 \left[\int_{I_n \cup J_n} e^{-\lambda t}f(A_n B^s_t) dt\right] \notag\\
	 &=\nu\sum_{n=1}^\infty \mathbf{E}_0 \left[\int_{I_n \cup J_n} e^{-\lambda t}f(B_t^s) dt\right] +(1-\nu) \sum_{n=1}^\infty \mathbf{E}_0 \left[\int_{I_n \cup J_n} e^{-\lambda t}f(- B_t^s) dt\right] \notag\\
	 &=\nu \mathbf{E}_0 \left[\int_{0}^\infty e^{-\lambda t}f(B_t^s) dt\right] +(1-\nu) \mathbf{E}_0 \left[\int_{0}^\infty e^{-\lambda t}f(- B_t^s) dt\right] \notag \\
	 &=\begin{cases}
	\nu \left(\frac{1}{\eta \Phi(\lambda) + \sqrt{\lambda}}\int_0^\infty e^{-y \sqrt{\lambda}} f(y) dy + \frac{\eta \Phi(\lambda)}{\lambda} \frac{f(0)}{\eta \Phi(\lambda) +\sqrt{\lambda}}\right) \quad y \geq 0\\
	 \\(1-\nu)\left(\frac{1}{\eta \Phi(\lambda) + \sqrt{\lambda}}\int_{-\infty}^0 e^{-y \sqrt{\lambda}} f(y) dy + \frac{\eta \Phi(\lambda)}{\lambda} \frac{f(0)}{\eta \Phi(\lambda) +\sqrt{\lambda}}\right) \quad y <0,
	 \end{cases}
	 \end{align}
	 where in the last equality we have used \eqref{resolvent: B^s} for the process $B^s$. We provide, by using Laplace transform, that the process we are considering, $B^{\nu, s}$, satisfies the conditions at zero \eqref{laplace:BC}. Indeed, we see, for $x>0, \, y>0$,
	 \begin{align*}
	R_\lambda f(x)= R_\lambda^D f(x) + e^{-\sqrt{\lambda} x} \nu \left(\frac{1}{\eta \Phi(\lambda) + \sqrt{\lambda}}\int_0^\infty e^{-y \sqrt{\lambda}} f(y) dy + \frac{\eta \Phi(\lambda)}{\lambda} \frac{f(0)}{\eta \Phi(\lambda) +\sqrt{\lambda}}\right)
	 \end{align*}
	 and, for $x<0, y>0$
	 \begin{align*}
	 R_\lambda f(x)= e^{\sqrt{\lambda} x} \nu \left(\frac{1}{\eta \Phi(\lambda) + \sqrt{\lambda}}\int_0^\infty e^{-y \sqrt{\lambda}} f(y) dy + \frac{\eta \Phi(\lambda)}{\lambda} \frac{f(0)}{\eta \Phi(\lambda) +\sqrt{\lambda}}\right).
	 \end{align*}
	 Then, for $y>0$, we have that \eqref{laplace:BC} is
	 \begin{align*}
	 &\eta \Phi(\lambda) R_\lambda f(0) - \eta \frac{\Phi(\lambda)}{\lambda} f(0) - \nu \frac{\partial}{\partial x} R_\lambda f(x) \big\vert_{x=0^+} +(1-\nu) R_\lambda f(x) \big\vert_{x=0^-}\\
	 &= \eta \Phi(\lambda) \left( \nu \frac{1}{\eta \Phi(\lambda) + \sqrt{\lambda}}\int_0^\infty e^{-y \sqrt{\lambda}} f(y) dy + \nu \frac{\eta \Phi(\lambda)}{\lambda} \frac{f(0)}{\eta \Phi(\lambda) +\sqrt{\lambda}} \right) -\nu \eta \frac{\Phi(\lambda)}{\lambda} f(0) + \\
	 \ &- \nu \left(\int_{0}^{\infty} e^{-\sqrt{\lambda y}} f(y) dy -\sqrt{\lambda}\nu \frac{1}{\eta \Phi(\lambda) + \sqrt{\lambda}}\int_0^\infty e^{-y \sqrt{\lambda}} f(y) dy - \sqrt{\lambda} \nu \frac{\eta \Phi(\lambda)}{\lambda} \frac{f(0)}{\eta \Phi(\lambda) +\sqrt{\lambda}}\right)+\\
	 \ &+(1-\nu) \left(\sqrt{\lambda}\nu \frac{1}{\eta \Phi(\lambda) + \sqrt{\lambda}}\int_0^\infty e^{-y \sqrt{\lambda}} f(y) dy + \sqrt{\lambda} \nu \frac{\eta\Phi(\lambda)}{\lambda} \frac{f(0)}{\eta \Phi(\lambda) +\sqrt{\lambda}}\right)\\
	 &=\eta \Phi(\lambda) \left( \nu \frac{1}{\eta \Phi(\lambda) + \sqrt{\lambda}}\int_0^\infty e^{-y \sqrt{\lambda}} f(y) dy + \nu \frac{\eta \Phi(\lambda)}{\lambda} \frac{f(0)}{\eta \Phi(\lambda) +\sqrt{\lambda}} \right) -\nu \eta \frac{\Phi(\lambda)}{\lambda} f(0) +\\
	 \ & - \nu \int_{0}^{\infty} e^{-\sqrt{\lambda y}} f(y) dy + \sqrt{\lambda}\nu \frac{1}{\eta \Phi(\lambda) + \sqrt{\lambda}}\int_0^\infty e^{-y \sqrt{\lambda}} f(y) dy + \sqrt{\lambda} \nu \frac{ \eta \Phi(\lambda)}{\lambda} \frac{f(0)}{\eta \Phi(\lambda) +\sqrt{\lambda}}\\
	 &= \left(\frac{\eta \Phi(\lambda) \nu + \sqrt{\lambda} \nu}{{\eta \Phi(\lambda) + \sqrt{\lambda}}} - \nu \right) \int_{0}^{\infty} e^{-\sqrt{\lambda y}} f(y) dy+ \nu \eta \frac{  \Phi(\lambda)}{\lambda} \frac{f(0)}{ \eta \Phi(\lambda) +\sqrt{\lambda}} \left(\eta \Phi(\lambda) + \sqrt{\lambda}\right) -\nu \eta \frac{\Phi(\lambda)}{\lambda} f(0) \\
	 &=0.
	 \end{align*}
	 Then, for positive paths, the  conditions at zero are satisfied. We verify the analogous for negative paths. 
	 Using the previous method, for $y<0$, we have
	 \begin{align*}
	&\eta \Phi(\lambda) R_\lambda f(0) - \eta \frac{\Phi(\lambda)}{\lambda} f(0) - \nu \frac{\partial}{\partial x} R_\lambda f(x) \big\vert_{x=0^+} +(1-\nu) R_\lambda f(x) \big\vert_{x=0^-}\\
	&=\eta \Phi(\lambda) \left(\frac{1-\nu}{\eta\Phi(\lambda) + \sqrt{\lambda}} \int_{-\infty}^0 e^{-\sqrt{\lambda} y} f(y) dy +  (1-\nu) \eta \frac{\Phi(\lambda)}{\lambda} \frac{f(0)}{\eta\Phi(\lambda) + \sqrt{\lambda}}\right) - (1-\nu) \eta \frac{\Phi(\lambda)}{\lambda} f(0) +\\
	\ &-\nu \left(-\sqrt{\lambda} \frac{1-\nu}{\eta\Phi(\lambda) + \sqrt{\lambda}} \int_{-\infty}^0 e^{-\sqrt{\lambda} y} f(y) dy - \sqrt{\lambda} (1-\nu) \eta \frac{\Phi(\lambda)}{\lambda} \frac{f(0)}{\eta\Phi(\lambda) + \sqrt{\lambda}}\right)+\\
	\ & +(1 - \nu) \left(- \int_{-\infty}^{0} e^{-\sqrt{\lambda} y} f(y) dy + \sqrt{\lambda} \frac{1-\nu}{\eta\Phi(\lambda) + \sqrt{\lambda}} \int_{-\infty}^0 e^{-\sqrt y} f(y) dy + \sqrt{\lambda} (1-\nu) \eta \frac{\Phi(\lambda)}{\lambda} \frac{f(0)}{\eta\Phi(\lambda) + \sqrt{\lambda}}\right)\\
	&= \eta \Phi(\lambda) \left(\frac{1-\nu}{\eta\Phi(\lambda) + \sqrt{\lambda}} \int_{-\infty}^0 e^{-\sqrt{\lambda} y} f(y) dy + (1-\nu) \eta \frac{\Phi(\lambda)}{\lambda} \frac{f(0)}{\eta\Phi(\lambda) + \sqrt{\lambda}}\right) - (1-\nu) \eta \frac{\Phi(\lambda)}{\lambda} f(0) +\\
	\ &+ \left(- (1-\nu)\int_{-\infty}^{0} e^{-\sqrt{\lambda} y} f(y) dy + \sqrt{\lambda} \frac{1-\nu}{\eta\Phi(\lambda) + \sqrt{\lambda}} \int_{-\infty}^0 e^{-\sqrt y} f(y) dy + \sqrt{\lambda} (1-\nu) \eta \frac{\Phi(\lambda)}{\lambda} \frac{f(0)}{\eta\Phi(\lambda) + \sqrt{\lambda}}\right)\\
	&= \left(\frac{(1-\nu) \, \eta \Phi(\lambda)}{\eta\Phi(\lambda) + \sqrt{\lambda}} + \frac{(1-\nu)\sqrt{\lambda}}{\eta\Phi(\lambda) + \sqrt{\lambda}} -(1-\nu)  \right) \int_{-\infty}^0 e^{-\sqrt{\lambda} y} f(y) dy + (1-\nu) \eta \frac{\Phi(\lambda)}{\lambda} f(0) \left(-1 + \frac{\eta \Phi(\lambda)+ \sqrt{\lambda}}{\eta\Phi(\lambda) + \sqrt{\lambda}}\right)\\
	&=0,
	 \end{align*}
	 as required. By inverting the Laplace transform, we conclude that $u(t,x)=\mathbf{E}_x\left[f(B^{\nu, s} _t)\right]$ satisfies the conditions at $x=0$ \eqref{laplace:BC}, which we are considering.
\end{proof}
\begin{remark}
	In the case of plateaus, the process $B^{\nu,s}$ behaves like a time-changed Brownian motion with the inverse of a subordinator with symbol $\Phi$, denoted by $B \circ L^\Phi$ in our notation. As we know, see for example \cite{cinlar}, the points at the end of the plateau are regenerative for which we can use the Markov property. 
\end{remark}
\begin{remark}
	We define the process $B^{\nu,s}$ \textit{non-local} sticky Brownian motion because  it can be seen as a generalization of the two-sided sticky Brownian motion. Indeed, for $\Phi(\lambda)=\lambda^\alpha$, $\alpha \in (0,1)$, and a function $f \in \mathcal{M}_0 \cap C[0,\infty)$ with $f^\prime \in \mathcal{M}_0$, we have
	\begin{align*}
	\int_0^\infty e^{-\lambda t} \mathfrak{D}_t^\Phi f(t) dt= \lambda^\alpha \widetilde{f}(\lambda) - \frac{\lambda^\alpha}{\lambda} f(0) \rightarrow \lambda \widetilde{f}(\lambda) - f(0) =\int_{0}^{\infty} e^{-\lambda t} f^\prime(t) dt \quad \text{as } \alpha \uparrow 1.
	\end{align*}
	Then, for $\alpha \uparrow 1$, the Caputo-Dzherbashian derivative tends to the first derivative and the conditions at zero become
	\begin{align*}
	\frac{\partial}{\partial t} u(t,0) - \nu \frac{\partial}{\partial x} u(t,0^+) + (1-\nu )\frac{\partial}{\partial x} u(t,0^-)=0
	\end{align*}
	which is a dynamic condition for the heat equation, that is the Feller-Wentzell conditions related to the sticky Brownian motion. For more details on dynamic conditions see \cite{romanelli}.
\end{remark}
\section*{Appendix}
For the convenience of the reader and to ensure that the work is self-contained, we present the following proof, which can be found in \cite[Formula 6, Section 15]{itomckean-halfline}.
\begin{lemma}
	Under the hypothesis of Theorem \ref{thm:non-local-skew}, for $\lambda>0$, we have
	\begin{align*}
	\mathbf{E}_0 \left[\int_{0}^\infty e^{-\lambda t}f(B_t^\bullet) dt\right]= \frac{\int_0^\infty R_\lambda^D f(y) \, \Pi^\Phi (dy)}{\Phi(\sqrt{\lambda})},
	\end{align*} 
	where $B^\bullet$ is defined in \eqref{Bpallino}, $R_\lambda^D$ is the resolvent of a killed Brownian motion at zero (with Dirichlet boundary conditions) and $\Phi$, with $\Pi^\Phi$, is given by \eqref{LevKinFormula}.
\end{lemma}
\begin{proof}
	We know that the process $H^\Phi L^\Phi$ that is a right-continuous process with jumps (given the way the subordinator $H^\Phi$ is constructed ), since $H^\Phi$ is right continuous with jumps and $L^\Phi$ is continuous. If we enumerate the jumps of $H^\Phi$ with $l_1, l_2,...$, we can decompose $(0,\infty)$ in two sets:
	\begin{align*}
	&\mathcal{J}^c:=\{t \geq 0 : H^\Phi L^\Phi = t\}\\
	&\mathcal{J} := \bigcup_{n \geq 1} \mathcal{J}_n= \bigcup_{n \geq 1} [l_n^{-}, l_n^+),
	\end{align*}
	where we have to include the point $l_n^-$ because $H^\Phi L^\Phi$ is right-continuous. From \eqref{Bpallino}, we see that the process outside from the jumps of $H^\Phi$ is a reflecting Brownian motion otherwise we have to consider the jump. In particular, we have
	\begin{align*}
	B^\bullet=
	\begin{cases}
	B^+ \quad & \gamma_t \in \mathcal{J}^c\\
	l_n^+-\gamma_t + B^+ \quad & \gamma_t \in \mathcal{J}_n.
	\end{cases}
	\end{align*}
	indeed if $\gamma_t \in \mathcal{J}_n$ we have $H^\Phi L^\Phi \gamma= l_n^+$. We observe that $\gamma_t \in [l_n^-,l_n^+)$ if and only if $t \in [\gamma_{t-}^{-1}(l_n^-), \gamma_{t-}^{-1}(l_n^+))$, in fact $\gamma_t$ is a continuous process and its right inverse $\gamma_t^{-1}$ is a right continuous process with jumps, then to make sure that the point $l_n^-$ is included we have to introduce
	\begin{align*}
	\gamma_{t-}^{-1}:=\inf \{s \geq 0 : \gamma_s \geq t\},
	\end{align*}  
	which is left continuous.  Then, the resolvent is
	\begin{align*}
	\mathbf{E}_{0}  \left[\int_0^\infty e^{-\lambda t} f (B^\bullet_t) dt\right]&= \mathbf{E}_{0}  \left[\int_{\mathcal{J}\cup \mathcal{J}^c} e^{-\lambda t} f(B^\bullet_t) dt\right]
	\\
	&=\mathbf{E}_{0}  \left[\int_{\mathcal{J}} e^{-\lambda t} f(B^\bullet_t) dt\right] + \mathbf{E}_{0}  \left[\int_{\mathcal{J}^c} e^{-\lambda t} f(B^\bullet_t) dt\right]\\
	&=\sum_{n \geq 1} \mathbf{E}_{0}  \left[\int_{\gamma_{t-}^{-1}(l_n^-)}^{\gamma_{t-}^{-1}(l_n^+)} e^{-\lambda t} f(B^\bullet_t) dt\right] + \mathbf{E}_{0}  \left[\int_{\mathcal{J}^c} e^{-\lambda t} f(B^\bullet_t) dt\right],
	\end{align*} 
	but we are dealing only with pure jump subordinators (with zero drift), then $\mathcal{J}^c$ has zero measure. For the remaining part, we have that
	\begin{align*}
	\mathbf{E}_{0}  \left[\int_0^\infty e^{-\lambda t} f(B^\bullet_t) dt\right]&=\sum_{n \geq 1} \mathbf{E}_{0}  \left[\int_{\gamma_{t-}^{-1}(l_n^-)}^{\gamma_{t-}^{-1}(l_n^+)} e^{-\lambda t} f(B^\bullet_t) dt\right]\\
	&=\sum_{n \geq 1}  \mathbf{E}_{0} \left[\int_0^{\gamma_{t-}^{-1}(l_n^+)-\gamma_{t-}^{-1}(l_n^-)} e^{-\lambda (s + \gamma_{t-}^{-1}(l_n^-))} f(B^\bullet_{s + \gamma_{t-}^{-1}(l_n^-)}) ds\right].
	\end{align*}
	We use the strong Markov property for $B^\bullet$, whose natural sigma algebra is denoted by $\mathcal{F}$, with respect to the stopping time $\gamma_{t-}^{-1}(l_n^-)$ and we get
	\begin{align*}
	&\mathbf{E}_{0}  \left[\int_0^\infty e^{-\lambda t} f(B^\bullet_t) dt\right]\\
	&=\sum_{n \geq 1} \mathbf{E}_{0}\left[ \mathbf{E}_{0} \left[\int_0^{\gamma_{t-}^{-1}(l_n^+)-\gamma_{t-}^{-1}(l_n^-)} e^{-\lambda (s + \gamma_{t-}^{-1}(l_n^-))} f(B^\bullet_{s + \gamma_{t-}^{-1}(l_n^-)}) ds\vert \mathcal{F}_{\gamma_{t-}^{-1}(l_n^-)}\right]\right]\\
	&=\sum_{n \geq 1} \mathbf{E}_{0} \left[\mathbf{E}_{0} \left[e^{-\lambda \gamma_{t-}^{-1}(l_n^-)} \mathbf{E}_{0}\left[ \int_{0}^{\gamma^{-1}(l_n)} e^{-\lambda t} f(B^\bullet_t) dt\right]\right]\right],
	\end{align*}
	but the process, during the jump, is $B^\bullet_t=l_n^+-\gamma_t + B^+$ and since $t \leq \gamma^{-1}(l_n) $ it behaves as $B_t$ for $t \leq \tau_0$ started at $l_n$, with $\tau_0$ hitting time at zero for $B_t$. Then, by using also $\gamma_{t-}^{-1}(l_n^-)=\gamma_{t}^{-1}(l_n^-)$ a.s. and its characteristic exponent, we obtain
	\begin{align*}
	\mathbf{E}_{0}  \left[\int_0^\infty e^{-\lambda t} f(B^\bullet_t) dt\right]&=\sum_{n \geq 1} \mathbf{E}_{0}\left[ \mathbf{E}_{0} \left[e^{-\lambda \gamma_{t}^{-1}(l_n^-)} \mathbf{E}_{l_n}\left[ \int_{0}^{\tau_0} e^{-\lambda t} f(B_t) dt \right]\right]\right]\\
	&=\sum_{n \geq 1} \mathbf{E}_{0} \left[e^{-\sqrt{\lambda} l_n^-} {R}_\lambda^D f(l_n)\right]\\
	&=\sum_{n \geq 1} \mathbf{E}_{0} \left[e^{-\sqrt{\lambda} H^\Phi (l_n^-)} {R}_\lambda^D f(l_n)\right]\\
	&=\mathbf{E}_{0} \left[\int_{(0,\infty) \times (0,\infty)} e^{-\sqrt{\lambda} H^\Phi (l^-)} {R}_\lambda^D f(l) N(dt \times dl)\right],
	\end{align*}
	where $N(dt \times dl)$ is the random measure associated to $H^\Phi$. From \cite[Example II.4.1]{ikeda2014stochastic} we know that $N(dt \times dl)= dt \, \Pi^\phi(dl)$, then we get
	\begin{align*}
	&\mathbf{E}_{0} \left[\int_{(0,\infty) \times (0,\infty)} e^{-\sqrt{\lambda} H^\Phi (l^-)} {R}_\lambda^D f( l) N(dt \times dl)\right]\\
	&\lim_{\varepsilon \to 0} \mathbf{E}_{0} \left[\int_{(0,\infty) \times (0,\infty)} e^{-\sqrt{\lambda} H^\Phi (l-\varepsilon)} {R}_\lambda^D f( l) N(dt \times dl)\right]
	\\
	&=\int_0^\infty \int_0^\infty e^{-t \Phi(\sqrt{\lambda})} {R}_\lambda^D f( l) dt \,\Pi^\Phi(dl)\\
	&=\frac{\int_0^\infty {R}_\lambda^D f(y) \Pi^\Phi(dl)}{\Phi(\sqrt{\lambda})}
	\end{align*}
	as required.
\end{proof}

\bibliographystyle{plain}
\bibliography{nlskewstick.bib}

\end{document}